\documentclass[reqno,11pt,a4paper]{amsart}
\usepackage{amsfonts}
\usepackage{layaureo}

\usepackage{dsfont,mathscinet}

\usepackage{mathtools}
\usepackage{graphicx}

\usepackage{epstopdf}
\usepackage[colorlinks=true, pdfstartview=FitV, linkcolor=blue, 
            citecolor=blue, urlcolor=blue]{hyperref}

\usepackage[T1]{fontenc}
\usepackage{lmodern}
\usepackage{microtype}
\usepackage{amsthm}
\theoremstyle{plain}
\newtheorem{theorem}{Theorem}[section]
\newtheorem{proposition}[theorem]{Proposition}
\newtheorem{lemma}[theorem]{Lemma}
\newtheorem{coro}[theorem]{Corollary}

\theoremstyle{definition}
\newtheorem{definition}[theorem]{Definition}
\newtheorem{remark}[theorem]{Remark}
\newtheorem*{remark*}{Remark}

\newtheorem*{hypo*}{Hypothesis}
\newtheorem{example}[theorem]{Example}

\usepackage{amsmath,amssymb,mathrsfs,mathtools,braket,latexsym,esint}
\numberwithin{equation}{section}

\def\nat{\mathbb{N}}

\def\real{\mathbb{R}}

\def\matr{\real^{3 \times 3}}

\def\cA{\mathcal{A}}
\def\cC{\mathcal{C}}

\def\cF{\mathcal{F}}
\def\cG{\mathcal{G}}

\def\rmT{\mathrm{T}}

\def\ny{\nabla y}

\def\nP{\nabla P}
\def\el{\mathrm{el}}

\def\hard{\mathrm{hard}}
\def\homo{\mathrm{hom}}

\def\SL{\mathsf{SL}}
\def\sl{\mathsf{sl}}
\def\GL{\mathsf{GL}}
\def\tr{\mathrm{tr}}
\def\io{\int_\Omega}

\def\ia{\int_A}

\def\dd{\,\mathrm{d}}

\def\R{\mathbb{R}}

\def\eps{\varepsilon}
\renewcommand{\epsilon}{\eps}

\usepackage{enumerate}
\usepackage[shortlabels]{enumitem}

\usepackage{color}

\usepackage{tikz}
\def\EEE{\color{black}}

\title[Homogenization in finite plasticity]{A homogenization result in finite plasticity}
\date{\today}

\author[E. Davoli]{Elisa Davoli}
\author[C. Gavioli]{Chiara Gavioli}
\author[V. Pagliari]{Valerio Pagliari}	

\AtEndDocument{
	\bigskip
	\footnotesize{
		\noindent
		E. Davoli, C. Gavioli, V. Pagliari.
		
		\noindent
		Institute of Analysis and Scientific Computing,
		TU Wien,
		
		\noindent
		Wiedner Hauptstra\ss e 8-10, 1040 Vienna, Austria.
		
		\noindent
		E-mails:
		\href{mailto:elisa.davoli@tuwien.ac.at}{\tt elisa.davoli@tuwien.ac.at},
		\href{mailto:chiara.gavioli@tuwien.ac.at}{\tt chiara.gavioli@tuwien.ac.at}, \href{mailto:valerio.pagliari@tuwien.ac.at}{\tt valerio.pagliari@tuwien.ac.at}.
	}
}

\begin{document}
	
	\begin{abstract}
		We carry out a variational study for integral functionals
		that model the stored energy of a heterogeneous material
		governed by finite-strain elastoplasticity with hardening.
		Assuming that the composite has a periodic microscopic structure,
		we establish the $\Gamma$-convergence
		of the energies in the limiting of vanishing periodicity.
		The constraint that plastic deformations belong to $\SL(3)$ poses
		the biggest hurdle to the analysis,
		and we address it by regarding $\SL(3)$ as a Finsler manifold.
		
		
	\end{abstract}
	
	\maketitle
	
	{\parskip=0em \tableofcontents}

	\normalcolor

	\section{Introduction}
	
	This paper deals with the homogenization of some variational integrals
	related to the theory of finite plasticity.
	Precisely, we focus on those materials
	governed by finite-strain elastoplasticity with hardening.
	Starting from the microscopic level,
	we aim at retrieving effective macroscopic descriptions
	by means of a $\Gamma$-convergence approach.
	
	Let $\Omega \subset \real^3$ be an open, bounded, connected set with Lipschitz boundary, and let us denote by $\SL(3)$ the group of $3 \times 3$ real matrices with determinant equal to $1$. Our main result describes the limiting behavior of the functionals
	\begin{equation}\label{Feps}
	\cF_\eps(y,P) \coloneqq \int_\Omega W\left(\frac{x}{\eps},\ny(x) P^{-1}(x)\right) \dd x + \int_\Omega H\left(\frac{x}{\eps},P(x)\right) \dd x + \io |\nabla P(x)|^q \dd x,
	\end{equation}
	where $y \in W^{1,2}(\Omega;\real^3)$, $P \in W^{1,q}(\Omega;K)$ with $q > 3$ and $K\subset\SL(3)$ compact,
	and $W$ and $H$ encode, respectively, the periodic nonlinear elastic energy density and the hardening 
	of the elastoplastic material sitting in $\Omega$.
	The precise statement is contained in Theorem~\ref{stm:homo-fin-plast},
	and the exact mathematical setting of the analysis is detailed in Section~\ref{sec:math}.
	
	As we touched upon before, we work within the classical mathematical theory of elastoplasticity at finite strains (see, e.g., \cite{lubliner}). We assume that the elastic behavior of our sample $\Omega$ is independent of preexistent plastic distortions. This can be rephrased as the assumption that the deformation gradient $\nabla y$ associated with any deformation $y \colon \Omega\to \R^3$ of the body decomposes into an elastic strain and a plastic one. Within the framework of linearized elastoplasticity, the decomposition would have an additive nature; in the setting of finite plasticity, instead, a multiplicative structure is traditionally assumed \cite{kroner, Lee}. Following this approach, as in \cite{Mie02, Mie03}, we suppose that for any deformation $y\in W^{1,2}(\Omega;\R^3)$
	there exist an \textit{elastic strain} $F_\el \in L^2(\Omega;\matr)$ and
	a \textit{plastic strain} $P \in L^2(\Omega;{\rm SL}(3))$ such that
	$$\ny(x) = F_\el(x)P(x) \quad \text{for a.\,e.\ }x \in \Omega.$$
	Such multiplicative decomposition, recently justified in the setting of dislocation systems and crystal plasticity in \cite{conti.reina, conti.schloemerkamper.reina}, motivates the definitions of the energy functionals in \eqref{Feps}.
	Other constitutive models in finite plasticity have been also taken into account in the mathematical literature (see, e.g., \cite{davoli.francfort, grandi.stefanelli1, grandi.stefanelli2, naghdi}).
	We point out that the stored energy in \eqref{Feps} has already been considered in \cite{MieSt} in a variational setting. We refer to \cite{Mie03} for more general models in finite plasticity including hardening variables.
	Concerning the regularization via $\nabla P$ in \eqref{Feps},
	we note that it is common in engineering models and prevents the formation of microstructures, see \cite{BCHH,CHM}. We also refer to \cite{drs} for a discussion of its drawbacks and for alternative higher order regularizations.
	
	In the small strain regime, the homogenization of the plasticity equations has been undertaken in the series of works \cite{schweizer.veneroni, schweizer.haida1, schweizer.haida2} both in the periodic and in the aperiodic and stochastic settings. The framework of perfect plasticity in the case of a linear elastic response has been completely characterized in the seminal work \cite{francfort.giacomini}.
	See also \cite{hanke} for a homogenization result in the small strain regime with hardening.
	To the authors' knowledge, the only available results in the large strain setting are instead confined to the framework of crystal plasticity for stratified composites. In particular, we refer to \cite{cc,cc1} for the superlinear growth case, to \cite{cdf} for the linear growth scenario, and to \cite{cd} for some first homogenization analysis in the evolutionary setting.
	Theorem~\ref{stm:homo-fin-plast} fills a gap in the study of elastoplastic microstructures
	by providing a static homogenization result in the large strain regime.
	As an application of our contribution, in \cite{DGP2}
	the analysis of high-contrast elastoplastic materials is carried out.
		
	We conclude this introduction with a few words on the proofs.
	The equicoervicity of $\{\cF_\eps\}$ and the semicontinuity of the elastic term require uniform bounds on the plastic strains (cf.~\eqref{Pbound}, see also \cite{MieSt}). To enforce them, we prescribe the effective domain of the hardening to be a compact neighborhood of the identity in $\SL(3)$.
	The existence of an integral $\Gamma$-limit, instead, is based on the localization technique in the context of integral representation results. The biggest hurdle here is devising a version of the so-called \textit{fundamental estimate} that complies with the constraint that plastic deformations must belong to the compact neighborhood of the identity.
	Note that this poses an additional difficulty with respect to, e.g., \cite{BaMi,DFMK}, which deal with manifold-valued Sobolev functions.
	Indeed, not only we constrain our maps to take values in a manifold,
	but we also require them to range in a specific set.
	To accommodate for this, an adequate notion of convexity is needed. Therefore, we will regard $\SL(3)$ as a manifold endowed with a Finsler metric, and impose the compact neighborhood of the identity to be also geodesically convex.
	We refer to Section~\ref{sec:pre} for a review of the localization method and of the Finsler structure on $\SL(3)$.
	The characterization of the limiting energy density is obtained by a perturbative argument that grounds on standard homogenization results, see Theorem~\ref{stm:hom-classic}.
	
	\subsection*{Outline}
	The setup of our analysis and the main result, Theorem~\ref{stm:homo-fin-plast},
	is presented in Section~\ref{sec:math}.
	Section~\ref{sec:pre} contains some reminders about the technical tools to be used in the sequel.
	The proof of Theorem~\ref{stm:homo-fin-plast} is the subject of Section~\ref{sec:hom-fin-plast}, which we conclude with the discussion on compactness of bounded energy sequences and convergence of (almost) minimizers.
	
	
	\section{Mathematical setting and results}\label{sec:math}
	
	Throughout the paper, $\Omega$ is an open, bounded, and connected set
	with Lipschitz boundary in $\real^3$
	(the analysis would not change significantly
	if we settled the problem in $\real^d$ with $d=2$ or $d>3$).
	We use $\matr$ and $\real^{3\times 3 \times 3}$
	to denote real-valued $3\times 3$ and $3\times 3 \times 3$ tensors, respectively.
	We use the notation $I$ for the identity matrix.
	The symbol $|\,\cdot\,|$ is indiscriminately adopted
	for the Euclidean norms in  $\real^3$, $\matr$ and $\real^{3\times 3 \times 3}$.
	To deal with plastic strains, we recall the classical notation
		\begin{align*}
			\SL(3) \coloneqq \{F \in \real^{3 \times 3} : \det F = 1\}. 
		\end{align*}
	In Subsection~\ref{sec:Finsler} we will endow $\SL(3)$
	with a metric structure and regard it as a \textit{Finsler manifold}. 	
	
	If $A \subset \real^3$ is a measurable set,
	we will denote by $\mathcal{L}^3(A)$ its three-dimensional Lebesgue measure.
	
	The building block of our study is the following variational notion of convergence:
	\begin{definition}\label{def:Gammaconv}
		Let $X$ be a set endowed with a notion of convergence. We say that the family $\{\cG_\eps\}$,
		with $\cG_\eps \colon X \to [-\infty,+\infty]$, \textit{$\Gamma$-converges} as $\eps \to 0$ to $\cG\colon X \to [-\infty,+\infty]$ if for all $x \in X$ and all infinitesimal sequences $\{\eps_k\}_{k\in\nat}$ the following holds:
		\begin{enumerate}
			\item for every sequence $\{x_k\}_{k\in\nat} \subset X$ such that $x_k \to x$, we have
			$$\cG(x) \le \liminf_{k\to+\infty} \cG_{\eps_k}(x_k);$$
			\item there exists a sequence $\{x_k\}_{k\in\nat} \subset X$ such that $x_k \to x$ and
			$$\limsup_{k\to+\infty} \cG_{\eps_k}(x_k) \le \cG(x).$$
		\end{enumerate}
	\end{definition}
	When $X$ is equipped with a topology $\tau$,
	we sometimes use expressions such as $\Gamma(\tau)$-convergence or $\Gamma(\tau)$-limit
	to emphasize the underlying convergence for sequences in $X$.
	We elaborate on the connections between $\Gamma$-convergence and homogenization in Subsection~\ref{sec:Gammaconv}.
	In what follows, for notational convenience,
	we indicate the dependence on $\eps_k$ by means of the subscript $k$ alone,
	e.g., $\cF_k \coloneqq \cF_{\eps_k}$.
	
	We collect here the assumptions under which
	the $\Gamma$-convergence of the functionals in \eqref{Feps} will be proved.
	
	Let $ Q\coloneqq [0,1)^3$ be the periodicity cell. 	
	We assume that
	the elastic energy density $W\colon \real^3 \times \matr \to [0,+\infty]$ satisfies the following:
	\begin{enumerate}[label=\textbf{E\arabic*:},ref={E\arabic*}]
		\item\label{E1} It is a Carath\'eodory function such that
		$W(\,\cdot\,,F)$ is $Q$-periodic for every $F\in \matr$.
		\item\label{E-growth} It is $2$-coercive and has at most quadratic growth, i.e.,
		there exist $0 < c_1 \le c_2$ such that for a.\,e.\ $x \in \real^3$ and for all $F \in \matr$
		$$ c_1 |F|^2 \le W(x,F) \le c_2\left(|F|^2+1\right).$$
		\item\label{E-lip} It is $2$-Lipschitz:
		there exists $c_3 > 0$ such that for a.\,e.\ $x \in \real^3$ and for all $F_1,F_2 \in \matr$
		$$	|W(x,F_1) - W(x,F_2)| \le c_3 \left(1 + |F_1| + |F_2|\right)|F_1 - F_2|. $$
	\end{enumerate}
In \ref{E-growth}--\ref{E-lip} the exponent $2$ may be well replaced by some generic $p>1$. We do not pursue this direction just for notational convenience.

		Let us compare \ref{E1}--\ref{E-lip}
		with the common requirements for physically admissible elastic energy densities.
		Setting
		\[
			\GL_+(3) \coloneqq \{F \in \matr : \det F > 0\},
		\]
		a couple of natural conditions are $W \equiv +\infty$ on $\matr \setminus \GL_+(3)$, which rules out interpenetrations, and $\lim_{\det F \to 0} W(F) = +\infty$, which means that it should take infinite energy to squeeze a small block of material down to a point.
		Feasible elastic densities are also required to be frame indifferent
		and to attain their minimum, conventionally set to $0$, on the identity matrix $I$.
		
		In our analysis, it would be particularly challenging to preserve the constraint of positive determinant for deformation gradients.
		Therefore, we work in a simplified setting and consider the whole space of $3\times 3$ matrices as domain for $W$.
		Similarly, the quadratic growth from above in \ref{E-growth} rules out the blow up of the energy densities, but it is a rather common assumption in variational studies.
		Frame indifference, that is, $W(x,RF) = W(x,F)$ for a.\,e.\ $x \in \real^3$, for all $F \in \GL_+(3)$ and all rotations $R$, does not play any role in our analysis, thus we ignore it.
		Finally, up to adding a constant to $W$, the growth condition \ref{E-growth} is compatible with the requirement $\min W(x,\,\cdot\,) = W(x,I) = 0$ for a.\,e.\ $x \in \real^3$.
	
	As for the hardening functional,
	we assume that
	$H \colon \real \times \matr \to [0,+\infty]$
	meets the ensuing requirements:
	\begin{enumerate}[label=\textbf{H\arabic*:},ref={H\arabic*}]
		\item\label{H0} $H$ is a Carath\'eodory function such that $H(\,\cdot\,,F)$ is $Q$-periodic for every $F\in \matr$.
		\item\label{H1} Assume that a Finsler structure on $\SL(3)$ is assigned.
			For a.\,e.\ $x\in\Omega$, $H(x,F)$ is finite if and only if $F \in K$,
			where $K\subset \SL(3)$ is a geodesically convex, compact neighborhood of $I$.
		\item\label{H2} The restriction of $H(x,\,\cdot\,)$ to $K$ is Lipschitz continuous, uniformly in $x$.
	\end{enumerate}
	Let us spend some words on \ref{H1}.
	This rather strong hypothesis prescribes that
	the effective domain of $H(x,\,\cdot\,)$, namely the set
	$\left\{F \in \SL(3) : H(x,F) < +\infty\right\}$,
	is an $x$-independent compact set containing $I$. This kind of constraint has already appeared in the variational literature, see \cite[Formulas~(2.6a)--(2.6b)]{MieSt}. 
	A fundamental consequence of technical advantage is that
	uniform $L^\infty$-bounds on the plastic strains become available,
	provided the latter give rise to finite hardening.
	Indeed, if $K$ is as in \ref{H1}, then
	there exists $c_K>0$ such that
	\begin{equation}\label{Pbound}
		|F| + |F^{-1}| \le c_K
		\quad \text{for every }F \in K ,
	\end{equation}
	because $\SL(3)$ is by definition well separated from $0$.
	From such bound we infer that
	for any $F \in K$ and $G\in\matr$
	\begin{equation}\label{stm:Pbound}
		\left|G\right| = \left|G F^{-1} F\right| \le c_K\left| G F^{-1}\right|.
	\end{equation}
	These bounds will be employed several times in our analysis to obtain estimates on the deformation gradient $\nabla y$ starting from a control on the elastic strain $\nabla y P^{-1}$.
	
	The further requirement that 
	the effective domain $K$ is geodesically convex
	(in the Finsler sense we mentioned, see Subsection~\ref{sec:Finsler} for the details)
	roots in the localization argument that we adopt
	to establish the $\Gamma$-convergence (see Subsection~\ref{sec:Gammaconv}).
	We recall that
	a subset of a Finsler manifold is said to be geodesically convex
	if, for any couple of points in the set,
	there is a unique shortest path contained in the set
	that joins those two points.
	The existence of a compact set $K$
	complying with \ref{H1}
	is ensured by Proposition~\ref{stm:Bao} and Remark~\ref{stm:K-in-H2} below.
	
	We are now ready to state the main result of this paper, namely
	the homogenization formula for the functionals in \eqref{Feps}.
	Recalling that we chose $q > 3$, we work in the space $W^{1,2}(\Omega;\real^3)\times W^{1,q}(\Omega;\SL(3))$
	endowed with the topology $\tau$ characterized by
	\begin{equation}\label{eq:tau}
		(y_k,P_k) \stackrel{\tau}{\to} (y,P)
		\quad\text{if and only if}
		\quad
		\begin{cases}
			y_k \to y &\text{strongly in } L^2(\Omega;\real^3), \\[1 mm]
			P_k \to P &\text{uniformly}.
		\end{cases}
	\end{equation}
	The topology $\tau$ is the ``right one'' from a variational perspective in view of Corollary~\ref{cor:comp}.
	
	\begin{theorem}\label{stm:homo-fin-plast}
		Let $\cF_\eps$ be the functionals in \eqref{Feps},
		which we extend by setting 
		$$
			\cF_\eps(y,P)=+\infty
			\quad\text{on }
			\big[L^2(\Omega;\real^3)\times L^q(\Omega;\SL(3))\big]
			\setminus \big[W^{1,2}(\Omega;\real^3)\times W^{1,q}(\Omega;K)\big].
		$$
		If $W$ and $H$ satisfy \ref{E1}--\ref{E-lip}
		and \ref{H0}--\ref{H2}, respectively,
		then for all $y \in  L^2(\Omega;\real^3)$, $P \in L^q(\Omega;\SL(3))$
		the $\Gamma$-limit
		$$\cF(y,P) \coloneqq \Gamma(\tau)\mbox{-}\lim_{\eps \to 0} \cF_\eps(y,P)$$
		exists. We also have that
		$$\cF(y,P) = \left\lbrace\begin{aligned}
			\displaystyle\io \Big(W_\homo(\ny(x),P(x)) &+ H_\homo(P(x)) + |\nP(x)|^q\Big) \dd x \\
			& \text{if } (y,P) \in W^{1,2}(\Omega;\real^3)\times W^{1,q}(\Omega;K), \\[3pt]
			+\infty \qquad\qquad\qquad\qquad &\text{otherwise in } L^2(\Omega;\real^3)\times L^q(\Omega;\SL(3)),
		\end{aligned}\right.$$
		where $W_\homo \colon \matr \times K \to [0,+\infty)$ and $H_\homo \colon K \to [0,+\infty)$ are defined as
		\begin{gather*}
			W_\homo(F,G) \coloneqq \lim_{\lambda\EEE \to +\infty} \frac{1}{\lambda\EEE^3} \inf\left\lbrace \int_{(0,\lambda\EEE)^3} W\big(x,(F+\ny(x))G^{-1}\big) \dd x : y \in W^{1,2}_0((0,\lambda\EEE)^3;\real^3) \right\rbrace, \\
			H_\homo(F) \coloneqq \int_{Q} H(z,F) \dd z.
		\end{gather*}
	\end{theorem}

By standard $\Gamma$-convergence arguments, we obtain the following.
\begin{coro}\label{cor:comp}
	Let $\{\eps_k\}$ be an infinitesimal sequence.
	\begin{enumerate}[label=(\roman*)]
		\item Suppose that $\{(y_k,P_k)\}\subset L^2(\Omega;\real^3)\times L^q(\Omega;\SL(3))$ satisfies
		$$
		\|y_k\|_{L^2(\Omega;\real^3)}\le c, \qquad \cF_k(y_k,P_k) \le c
		$$
		for some $c>0$, uniformly in $k$. Then, there exist subsequences of $\{\eps_k\}$, $\{y_k\},$ and $\{P_k\}$, which we do not relabel, as well as $y \in W^{1,2}(\Omega;\real^3)$ and $P \in W^{1,q}(\Omega;K)$, such that $(y_k,P_k) \stackrel{\tau}{\to} (y,P)$.
		\item Let $\{(y_k,P_k)\}\subset W^{1,2}_0(\Omega;\real^3)\times W^{1,q}(\Omega;K)$ be a sequence of almost minimizers, that is,
		$$
		\lim_{k\to+\infty} \Big(\cF_k(y_k,P_k) - \inf \cF_k(y,P)\Big) = 0,
		$$
		where the infimum is taken over $W^{1,2}_0(\Omega;\real^3)\times W^{1,q}(\Omega;K)$. Then, there exists a minimizer $(y,P) \in W^{1,2}_0(\Omega;\real^3)\times W^{1,q}(\Omega;K)$ of $\cF$ such that, up to subsequences, $(y_k,P_k) \stackrel{\tau}{\to} (y,P)$. Moreover,
		$$
		\lim_{k\to+\infty} \Big(\inf \cF_k - \min \cF\Big) = 0.
		$$
	\end{enumerate}
\end{coro}

	
	\section{Preliminaries}\label{sec:pre}
	We gather in this section the technical tools
	to be employed in the sequel.
	
	\subsection{Localization and integral representation}\label{sec:Gammaconv}
	To the aim of laying the ground for the proof of Theorem~\ref{stm:homo-fin-plast},
	we briefly outline the localization technique
	in the context of integral representation results for $\Gamma$-limits.
	More detail and a thorough treatment of $\Gamma$-convergence,
	which we introduced in Definition~\ref{def:Gammaconv},
	may instead be found in the monographs \cite{BrDFr,DalM,braides}.
	For definiteness, we report on a well-known result
	(see, e.g., \cite[Theorem~14.5]{BrDFr}) that is underpinned by the localization method.
	
	\begin{theorem}\label{stm:hom-classic}
	Let $g\colon \real^3 \times \matr \to [0,+\infty)$ be a Borel  function
	that is $Q$-periodic in its first argument and
	that satisfies standard $p$-growth conditions for some $p \in (1,+\infty)$.
	For $\eps>0$ and $y \in L^p(\Omega;\real^3)$
	we define
	\[
		\cG_\eps(y) \coloneqq 
				\begin{cases}
				\displaystyle{
					\io g\bigg(\frac{x}{\eps},\ny(x)\bigg) \dd x
					}
				& \text{if } y \in  W^{1,p}(\Omega;\real^3), \\[3pt]
				+\infty
				& \text{otherwise}.
				\end{cases}
	\]
	Then, we have
	\[
		\Gamma(L^p)\mbox{-}\lim_{\eps\to 0} \cG_\eps(y)
		= 
					\begin{cases}
					\displaystyle{
						\io g_\homo\big(\ny(x)\big) \dd x
						}
					& \text{if } y \in  W^{1,p}(\Omega;\real^3), \\[3pt]
					+\infty
					& \text{otherwise},
					\end{cases}
	\]
	where the $\Gamma$-limit is taken with respect to the strong $L^p(\Omega;\mathbb{R}^3)$-topology, \EEE and
	$g_\homo \colon\matr \to [0,+\infty)$ is a quasiconvex function
	characterized by the asymptotic homogenization formula
	$$
		g_\homo(F) = \lim_{\lambda\EEE \to +\infty} \frac{1}{\lambda\EEE^3} \inf \left\lbrace \int_{(0,\lambda\EEE)^3} g\big(x,F+\ny(x)\big) \dd x : y \in W^{1,p}_0((0,\lambda\EEE)^3;\real^3) \right\rbrace
	$$
	for all $F \in \matr$.
	\end{theorem}
	
	As starting point to establish this theorem,
	one \EEE resorts to a general property of $\Gamma$-convergence
	which ensures that
	if $X$ is a separable metric space,
	then any family $\{\cG_\eps\}$
	with $\cG_\eps \colon X \to [-\infty,+\infty]$
	has a $\Gamma$-convergent subsequence
	(see, e.g., \cite[Proposition~7.9]{BrDFr}).	
	Such $\Gamma$-compactness principle yields that,
	up to subsequences, 
	an abstract $\Gamma(L^p)$-limit of $\{\cG_\eps\}$ exists.
	One is then naturally led to wonder
	whether the limit is in turn a functional of integral type.
	In order to show that this is in fact the case,
	a \textit{localization argument} is employed.
	This amounts to regard 
	$\cG_\eps$ as a function of the pair $(y,A)$,
	where $y \in L^p(\Omega;\real^3)$ and
	$A \in \cA(\Omega) \coloneqq \{\text{open subsets of }\Omega\}$; more precisely, setting
	$$
		\cG_\eps(y,A) \coloneqq \int_{A} g\bigg(\frac{x}{\eps},\ny(x)\bigg) \dd x
		\qquad
		\text{if } y \in  W^{1,p}(A;\real^3)
	$$
	and $\cG_\eps(y,A) \coloneqq +\infty$ otherwise, the idea is to focus on the properties of $\cG_\eps(y,\,\cdot\,)$ as a set function.
	In this respect, we recall some terminology.
	
	\begin{definition}\label{setfunct}
	Let $\alpha\colon \cA(\Omega) \to [0,+\infty]$ be a set function.
	We say that $\alpha$ is
	\begin{itemize}
	\item an increasing set function
		if $\alpha(\emptyset) = 0$ and $\alpha(A) \le \alpha(B)$ if $A \subseteq B$;
	\item subadditive
		if $\alpha(A \cup B) \le \alpha(A) + \alpha(B)$ for all $A,B$;
	\item superadditive
		if $\alpha(A \cup B) \ge \alpha(A) + \alpha(B)$ 
		for all $A,B$ such that $A \cap B = \emptyset$;
	\item inner regular
		if $\alpha(A) = \sup\left\lbrace \alpha(B) : B \in \cA(\Omega), \, B \Subset A \right\rbrace$.
	\end{itemize}
	\end{definition}
	The De Giorgi-Letta criterion
	(see, e.g., \cite[Theorem~10.2]{BrDFr})	states that
	an increasing set function is a restriction to $\cA(\Omega)$ of a Borel measure
	if and only if
	subadditive, superadditive and inner regular.
	Therefore,
	since the $\Gamma(L^p)$-limit of $\{\cG_\eps\}$ must be an increasing set function
	and we work under $p$-growth conditions,
	the integral representation boils down to proving that
	$\Gamma(L^p)\mbox{-}\lim \cG_\eps$ is subadditive, superadditive and inner regular.
	The subadditivity and inner regularity are the most delicate points, and
	they hinge in turn upon the so-called
	\textit{fundamental estimate}
	(see, e.g., \cite[Definition~11.2]{BrDFr}).
	Roughly speaking,
	given $A, B \in \cA(\Omega)$, $y \in L^p(\Omega;\real^3)$
	and sequences $\{y_\eps'\}$ and $\{y_\eps''\}$ converging to $y$ in $L^p$ such that
	$$
		\Gamma(L^p)\mbox{-}\lim \cG_\eps(y,A) = \lim_{\eps\to0} \cG_\eps(y_\eps',A)
		\quad \text{and} \quad
		\Gamma(L^p)\mbox{-}\lim \cG_\eps(y,B) = \lim_{\eps\to0} \cG_\eps(y_\eps'',B),
	$$	
	the fundamental estimate allows to construct a third sequence $\{y_\eps\}$ converging to $y$ in $L^p$ and such that
	$$
		\cG_\eps(y_\eps,A \cup B) \le \cG_\eps\EEE(y_\eps',A) + \cG_\eps(y_\eps'',B) + R_\eps,
		\qquad
		\lim_{\eps\to 0} R_\eps = 0,
	$$
	with $R_\eps$ satisfying a precise estimate in terms of $\|y_\eps'-y_\eps''\|_{L^p(A \cap B)}$.
	The exact form of the fundamental estimate to be used in our analysis is contained in \eqref{FE},
	while its application to establish subadditivity and inner regularity is discussed in Proposition~\ref{sub}.
	
	\subsection{Finsler structure on $\SL(3)$}\label{sec:Finsler}
	To the purpose of devising a form of the fundamental estimate
	that suits the functionals under consideration,
	it is convenient to endow $\SL(3)$ with a metric structure.
	In order to link the latter to the physics of the system we would like to model,
	we follow the approach in \cite{Mie02},
	which is grounded on the concept of \emph{plastic dissipation}.
	
	We recall some basic facts about the geometry of $\SL(3)$,
	which is a smooth manifold
	with respect to the topology induced by the inclusion in $\matr$.
	For every $F \in \SL(3)$ the tangent space at $F$ is characterized as
	$$
	\rmT_F \SL(3)=F{\sf sl}(3) \coloneqq \{FM \in \matr : \tr M = 0\}.
	$$
	In particular, $\rmT_I \SL(3)$ coincides with
	${\sf sl}(3)\coloneqq \{M \in \matr : \tr M = 0\}$.
	We equip $\SL(3)$ with a Finsler structure
	starting from a function $\Delta_I \colon {\sf sl}(3)\to [0,+\infty)$
	on which we make the following assumptions (cf.~\cite[Section~1.1]{BaChSh} and \cite[Section~1]{Mie02}):
	\begin{enumerate}[label=\textbf{D\arabic*:},ref={D\arabic*}]
		\item\label{D0} It is $C^2$ on ${\sf sl}(3)\setminus\{0\}$;
		\item\label{D1} It is positively $1$-homogeneous:
		$\Delta_I(c M) = c\Delta_I(M)$ for all $c > 0$ and $M \in {\sf sl}(3)$;
		\item\label{D3} The function $\Delta_I^2/2$ is strongly convex;
		\item\label{D2} It is $1$-coercive and has at most linear growth:
		there exist $0 < c_4 \le c_5$ such that for all $M \in {\sf sl}(3)$
		$$
		c_4|M| \le \Delta_I(M) \le c_5|M|.
		$$
	\end{enumerate}
	
	We point out that
	we work under stronger regularity assumptions than the ones in \cite{Mie02}.
	This is mainly due to the fact that
	we borrow results developed within the context of differential geometry,
	where smoothness is customarily required.
	As a consequence, some models,
	such as single crystal plasticity,
	are not covered by our analysis;
	on the other hand, our assumptions encompass Von Mises plasticity,
	see the considerations in Example~\ref{ex:VonMises}
	or \cite{HaMiMi,Mie02} for a wider discussion.
	
	We also recall that
	Finsler structures are known to appear as
	homogenized limits of periodic Riemannian metrics \cite{AB}.
	
	Essentially, $\Delta_I$ is a Minkowski norm on ${\sf sl}(3)$,
	which we can ``translate'' to the other tangent spaces by setting
	\begin{equation*}
		\begin{array}{rccc}
			\Delta\colon & \rm T\SL(3) &\to& [0,+\infty)\\[1pt]
			& (F,M) &\mapsto& \Delta_I(F^{-1}M),
		\end{array}
	\end{equation*}
	where $ \rm T\SL(3)$ is the tangent bundle to $\SL(3)$.
	Then, it can be checked that $(\SL(3),\Delta)$ is a $C^2$ Finsler manifold.
	We refer to the monograph \cite{BaChSh} by {\sc Bao, Chern \& Shen}
	for an introduction to Finsler geometry.
		
	Let now $\mathcal C(F_0,F_1)$ be the family of piecewise $C^2$ curves
	$\Phi\colon [0,1] \to\SL(3)$
	such that $\Phi(0) = F_0$ and $\Phi(1) = F_1$.
	We define a non-symmetric distance on $\SL(3)$ as follows:
	\begin{equation}\label{eq:D}
		D(F_0,F_1) \coloneqq
		\inf\left\lbrace
		\int_0^1 \Delta\big( \Phi(t),\dot{\Phi}(t) \big) \dd t 
		: \Phi \in \mathcal C(F_0,F_1)
		\right\rbrace,
	\end{equation}
	where $\dot{\Phi}$ is the velocity of the curve.
	The function $D$ is positive,
	attains $0$ if and only if it is evaluated on the diagonal of $\SL(3)\times\SL(3)$,
	and fulfills the triangle inequality;
	in general, however, $D(F_0,F_1)\neq D(F_1,F_0)$.
	
	From a physical viewpoint, if $P_0,P_1\colon \Omega \to \SL(3)$ are admissible plastic strains,
	the integral of $D(P_0,P_1)$ over $\Omega$ is interpreted 
	as the minimum amount of energy
	that is dissipated when the system moves from a plastic configuration to another.
	
	Note that under assumptions \ref{D0}--\ref{D3} it follows that $\Delta$ is subadditive (see \cite[Theorem~1.2.2]{BaChSh}), hence convex. Therefore, by	an application of the direct method of the calculus of variations (cf.~\cite[Theorem~5.1]{Mie02})
	it can be proved that for every $F_0,F_1\in \SL(3)$ there exists a curve
	$\Phi \in C^{1,1}([0,1];\SL(3))$ such that $\Phi(0) = F_0$, $\Phi(1) = F_1$ and
	$$
	D(F_0,F_1) = \int_0^1 \Delta\big(\Phi(t),\dot{\Phi}(t)\big) \dd t.
	$$
	We call such $\Phi$ a shortest path between $F_0$ and $F_1$.
	The following result,
	which summarizes the content of
	\cite[Exercise~6.3.3]{BaChSh},
	is crucial for the proof of the fundamental estimate.
	
	\begin{proposition}\label{stm:Bao}
		Assume that \ref{D0}--\ref{D3} hold.
		For any point $F$ in the Finsler manifold $\SL(3)$
		there exists a relatively compact neighborhood $U$ of $F$ such that
		for any $F_0,F_1\in U$ there exists a unique shortest path $\Phi$
		joining $F_0$ and $F_1$, and
		such path depends smoothly on its endpoints $F_0$ and $F_1$.
	\end{proposition}
	
The smooth dependence of shortest paths on their endpoints follows from the fact that each of them is a geodesic, and thus a solution of an ODE.
We recall that a path between $F_0$ and $F_1$ is called a geodesic
if it is a critical point of the length functional
under variations that do not change the endpoints, and that,
given $(F,M)\in T\SL(3)$ such that $\Delta(F,M)$ is sufficiently small,
there exists a unique geodesic $\gamma_{F,M}\colon (-2,2) \to \SL(3)$
satisfying $\gamma_{F,M}(0)=F$ and $\dot\gamma_{F,M}(0)=M$.
Therefore, we have (see \cite[p.\,126]{BaChSh}) that, for $M$ in a neighborhood of $0\in T_F\SL(3)$, we can define the so-called exponential map as
\begin{equation*}
	\exp(F,M) \coloneqq \begin{cases}
	F 							& \text{if  } M=0,\\
	\gamma_{F,M}(1) & \text{otherwise}.
	\end{cases}
\end{equation*}
For each $F \in \SL(3)$,
$\exp(F,\,\cdot\,)$ is a $C^1$-diffeomorphism in a neighborhood of $0 \in F\sl(3)$
(see \cite[Section~5.3]{BaChSh}).
In particular, there $\exp(F,\,\cdot\,)$ is invertible.
Denoting by $\exp_F^{-1}$ its inverse,
if $\Phi$ is the unique shortest path from $F_0$ to $F_1$ (Proposition~\ref{stm:Bao}),
one sees that necessarily 
\begin{equation}\label{eq:shortest-exp}
	\Phi(t)=\exp(F_0,t\exp_{F_0}^{-1}(F_1)), \quad t \in [0,1].
\end{equation}
	
	We conclude this section by noting that, while \ref{D0}--\ref{D3} characterize the Finslerian structure in an intrinsic way, \ref{D2} instead relates the Minkowski norm to the Euclidean one in the ambient space. In general, when such an assumption is made,
	the Finslerian distance 
	between two points of an embedded Finslerian manifold
	is locally comparable to the Euclidean one.
	A lower bound on the Finslerian distance in terms of the Euclidean one
	may be retrieved by relying on the first inequality in \ref{D2} and on the fact that shortest paths in $\real^{3\times3}$ are segments.
	Since such a bound is not needed in the sequel, in the next lemma we present only the upper one.
	The proof is based on the implicit function theorem. For notational convenience, we state it just in the specific case of $(\SL(3),\Delta)$, which is an embedded $8$-dimensional submanifold of $\real^{3\times 3}$.

	\begin{lemma}\label{stm:DvsEucl}
		Let \ref{D0}--\ref{D2} hold, and let $D$ be as in \eqref{eq:D}.
		For every $F\in \SL(3)$
		there exist a bounded open neighborhood $U$ of $F$ and
		a constant $c>0$ such that
		$$
		D(F_0,F_1) \le c\,|F_1-F_0|
		$$
		for any $F_0,F_1 \in U$.
	\end{lemma}

\begin{proof}
	Observe that, by definition of $D$ and $\Delta$
	and by employing the upper bound in \ref{D2},
	we have
	\begin{align}
	\nonumber		D(F_0,F_1) &\leq \int_0^1 \Delta\big( \Phi(t),\dot{\Phi}(t) \big) \dd t \\
	\nonumber							& = \int_0^1 \Delta_I\big( \Phi(t)^{-1}\dot{\Phi}(t) \big) \dd t \\
	\label{eq:DleqEu}				& \leq c_5 \int_0^1 \big| \Phi(t)^{-1}\dot{\Phi}(t) \big| \dd t 
	\end{align}
	for any piecewise $C^2$ curve joining $F_0$ and $F_1$.
	In order to prove the desired estimate,
	we need to choose $\Phi$ suitably.
	
	Since $\SL(3)$ is a submanifold of $\real^{3\times 3}$,
	we can locally represent it as a graph
	defined on the affine tangent space at a given point.
	To be precise, for any $F\in \SL(3)$,
	let $\pi_F$ be the affine tangent hyperplane to $\SL(3)$ at $F$,
	$\hat{n}_F\in \real^{3\times 3}$ be a unit normal to $\SL(3)$ at $F$, and 
	$B_r(F)\subset \real^{3\times 3}$ be the Euclidean ball
	centered at $F$ with radius $r>0$.
	Then, there exist
	a number $\delta>0$, a bounded, open neighborhood $U$ of $F$
	and a smooth function $\psi\colon \pi_F \cap \overline{B_\delta(F)}\to \real$
	such that
	$$
	U = \left\{ G\in \real^{3\times 3} : G=F+M+\psi(M)\hat{n}_F,\ M \in \pi_F \cap B_\delta(F)\right\}.
	$$

	Now, given $F\in \SL(3)$ and the
	associated $\delta$, $U$ and $\psi$,
	consider $F_0,F_1\in U$ and the matrices $M_0,M_1 \in \pi_F \cap B_\delta(F)$
	such that
	\begin{equation}\label{eq:Fi}
	F_i = F+M_i+\psi(M_i)\hat{n}_F \quad \mbox{for } i=0,1.
	\end{equation}
	Since $\pi_F \cap B_\delta(F)$ is convex, the curve
	$$
	\overline\Phi(t)\coloneqq F + M_0 + t(M_1-M_0)+ \psi \Big( M_0 + t(M_1-M_0) \Big)\hat{n}_F,
	\quad t\in [0,1]
	$$
	is a path in $U$ joining $F_0$ and $F_1$.
	We now choose $\Phi=\overline \Phi$ in \eqref{eq:DleqEu}.
	First, we observe that,
	since $U$ is bounded and $\overline \Phi(t)\in\SL(3)$ for all $t$,
	there exist a $U$-dependent constant $c>0$
	such that $|\overline\Phi^{-1}(t)|\leq c$ for all $t$
	(cf.~\eqref{Pbound}).
	Therefore, from \eqref{eq:DleqEu} we infer
	\begin{align*}
	D(F_0,F_1) &\leq c \int_0^1 \big| \dot{\overline\Phi}(t) \big| \dd t \\
	& \leq c \int_0^1 \Big| (M_1 - M_0)\Big(I+\nabla \psi\big( M_0 + t(M_1-M_0) \big)\Big)  \Big| \dd t \\
	& \leq c\, | M_1 - M_0 |,
	\end{align*}
	where we used that $\nabla\psi$ is continuous on $\pi_F \cap \overline{B_\delta(F)}$. To conclude, it suffices to observe that, by \eqref{eq:Fi}, $| M_1 - M_0 | \leq c\, | F_1 - F_0 |$.
\end{proof}

\begin{remark}\label{stm:K-in-H2}
	Grounding on Proposition~\ref{stm:Bao} and Lemma~\ref{stm:DvsEucl}, we can show that
	there exists a compact $K$ meeting the requirements in \ref{H1}.
	Let $U$ be a relatively compact neighborhood of $I\in\SL(3)$
	such that for any $F_0,F_1\in U$ there is a unique shortest path $\Phi$
	joining $F_0$ and $F_1$. Up to restricting to a subset, we can suppose that $U$ is such that Lemma~\ref{stm:DvsEucl} holds.
	By a Finsler variant of a theorem by Whitehead \cite[Exercise~6.4.3]{BaChSh},
	there exists an open neighborhood $V$ of $I$
	that is compactly contained in $U$ and
	geodesically convex.
	Since $K\coloneqq \bar V\subset U$,
	there is a unique shortest path $\Phi$
	from $F_0$ to $F_1$ for any  $F_0,F_1\in K$.
	By \eqref{eq:shortest-exp}, the path depends smoothly on its endpoints.
	The fact that $K$ is geodesically convex as well follows then by the same argument that proves that the closure of a convex set is still convex.
\end{remark}

	\begin{example}[Von Mises plasticity]\label{ex:VonMises}
		Von Mises plasticity is a model for polycrystals, such as metals,
		where the presence of grains averages out the different crystallographic directions.
		As detailed in \cite[Section~4.1]{Mie02},
		in such model the Minkowski norm $\Delta_I$ is given
		by a scalar product on the tangent space $\sl(3)$.
		This is equivalent to assuming that $\SL(3)$ is equipped with a left-invariant Riemannian metric. Exploiting the properties of $\sl(3)$, $\Delta_I$ can be put into the general form (see \cite[Formula~(5)]{HaMiMi} or \cite[Formula~(4.3)]{Mie02})
		$$
		\Delta_I(M) = \left(\beta_1|M_{\rm sym}|^2 + \beta_2|M_{\rm anti}|^2\right)^{1/2}\ \mbox{with } M_{\rm sym} = \frac{M + M^T}{2}, \ M_{\rm anti} = \frac{M - M^T}{2},
		$$
		for some positive material parameters $\beta_1$ and $\beta_2$.
		Here $|\,\cdot\,|$ denotes the matrix Frobenius norm, that is,
		the (element-wise) Euclidean norm $|\,\cdot\,|$ of $\real^{3\times 3}$.
		Note that $\Delta_I$ does satisfy assumptions \ref{D0}--\ref{D2}. By standard results in Riemannian geometry (see, e.g., \cite[Chapter~3, Proposition~4.2]{DoC}), for any matrix in $\sl(3)$ there exists a number $r > 0$ such that the closure of the geodesic ball of radius $r$ is geodesically convex. We recall that geodesic balls are defined as the images through the exponential map of balls centered at the origin on the tangent space, provided the exponential map is there a diffeomorphism (see \cite[p.\,70]{DoC}). In conclusion, for the case of Von Mises plasticity, the set $K$ in \ref{H1} can be chosen as the closure of a suitable geodesic ball centered at the identity.
	\end{example}
	
	
	\section{Homogenization in finite plasticity}\label{sec:hom-fin-plast}
	We devote this section to the proof of Theorem~\ref{stm:homo-fin-plast},
	that is, we exhibit the $\Gamma(\tau)$-limit of the functionals in \eqref{Feps}.
	We recall that the topology $\tau$ was introduced in \eqref{eq:tau}.
	
	\begin{remark}\label{stm:conv-inv}
		The energy functionals at stake depend on the plastic strain $P$ through its inverse $P^{-1}$.
		In this respect, it is useful to notice that
		if $P_k \to P$ uniformly,
		then $P_k^{-1} \to P^{-1}$ uniformly as well.
		Indeed, recalling that for any $k\in\nat$ we can write
		\begin{equation*}
			P_k^{-1} = \frac{\left({\rm cof}P_k\right)^T}{\det P_k} = \left({\rm cof}P_k\right)^T,
		\end{equation*}
		we deduce convergence for the sequence of inverses.	
	\end{remark}
	
	Since $(W^{1,2}(\Omega;\real^3)\times W^{1,q}(\Omega;\SL(3)),\tau)$ is a separable metric space,
	we know from general properties of $\Gamma$-convergence that,
	up to extraction of subsequences,
	the functionals in \eqref{Feps} $\Gamma(\tau)$-converge.
	Our main task is therefore to show that the limit is an integral functional,
	and we achieve this by a localization approach
	as the one streamlined in Subsection~\ref{sec:Gammaconv}.
	Precisely, denoting by $\cA(\Omega)$ the family of open subset of $\Omega$,
	with a slight abuse of notation,
	for any triple $(y,P,A)\in L^2(A;\real^3)\times L^q(A;\SL(3)) \times \cA(\Omega)$
	we set 
	\begin{equation}\label{FepsA}
		\cF_k(y,P,A) \coloneqq
		\left\lbrace
		\begin{aligned}
			\displaystyle \ia W\left(\frac{x}{\eps_k},\ny P^{-1}\right) \dd x &+
			\displaystyle \ia H\left(\frac{x}{\eps_k},P\right) \dd x +  \ia |\nabla P|^q \dd x \\[6pt]
			& \text{if } (y,P) \in W^{1,2}(A;\real^3)\times W^{1,q}(A;K), \\[3pt]
			+\infty \quad\qquad\qquad\qquad\qquad&\text{otherwise in } L^2(A;\real^3)\times L^q(A;\SL(3)),
		\end{aligned}
		\right.
	\end{equation}
	where ${\eps_k}$ is an infinitesimal sequence.
	
	We first show \EEE that
	the limits of $\Gamma$-convergent subsequences are in turn integral functionals.
	Second, we characterize the limiting energy densities,
	proving as well that 
	the whole family $\{\cF_\eps\}$ $\Gamma$-converges.
	
	\subsection{Integral representation}
	In this subsection we establish the following:
	
	\begin{theorem}\label{stm:homo-fin-plast-bis}
		Let $\cF_k$ be as in \eqref{FepsA}, where $W$ and $H$ satisfy \ref{E1}--\ref{E-lip} and \ref{H0}--\ref{H2}, respectively.
		Then, up to subsequences, $\{\cF_k\}$ $\Gamma(\tau)$-converges and
		for all $y \in L^2(A;\real^3)$ and $P \in L^q(A;\SL(3))$
		\begin{equation}\label{Fint}
			\Gamma(\tau)\mbox{-}\lim_{k\to+\infty} \cF_k(y,P,A) = 
			\begin{cases}
				\displaystyle \ia f(x) \dd x & 
				\text{if } (y,P) \in W^{1,2}(A;\real^3)\times W^{1,q}(A;K), \\[6pt]
				+\infty &\text{otherwise in } L^2(A;\real^3)\times L^q(A;\SL(3)),
			\end{cases}
		\end{equation}	
		for some $f \in L^1_{\rm loc}(\real^3)$
		(which depends on the subsequence).
	\end{theorem}

	As a first step, we introduce a version of the fundamental estimate
	fit for the functionals in \eqref{FepsA}.
	For our purposes (cf.~Proposition~\ref{almsub}), it suffices to formulate it for functions in $W^{1,2}(A;\real^3) \times W^{1,q}(A;K)$, which is the effective domain of the functionals $\cF_k(\,\cdot\,,\,\cdot\,,A)$ in \eqref{FepsA} for $A \in \cA(\Omega)$.
	We recall that,
	given $A,A' \in \cA(\Omega)$ with $A' \Subset A$
	(i.e., $\overline{A'}$ is a compact set contained in $A$),
	we say that a function $\varphi$ is a {\em cut-off function}
	between $A'$ and $A$ if $\varphi \in C^\infty_0(A)$,
	$0 \le \varphi \le 1$ and $\varphi \equiv 1$ in a neighborhood of $\overline{A'}$.
	
	\begin{definition}
		Let $\cC$ be a class of functionals $\cF\colon W^{1,2}(\Omega;\real^3) \times W^{1,q}(\Omega;K) \times \cA(\Omega) \to [0,+\infty]$. We say that $\cC$ satisfies uniformly the {\em fundamental estimate}
		if for every $A,A',B \in \cA(\Omega)$ with $A' \Subset A$ and for every $\sigma > 0$ there exists a constant $M_\sigma > 0$ with the following property: for all $\cF \in \cC$ and for every $(y_1,P_1),(y_2,P_2) \in W^{1,2}(\Omega;\real^3) \times W^{1,q}(\Omega;K)$
		there exist a cut-off function $\varphi$ between $A'$ and $A$
		and a path $\gamma\colon [0,1] \times K \times K \to K$ satisfying $\gamma(0,F,G)=F$, $\gamma(1,F,G)=G$ for all $F,G \in K$ such that
		\begin{equation}\label{FE}
			\begin{aligned}
				&\cF\big(\varphi y_1+(1-\varphi)y_2,\gamma\circ(\varphi,P_2,P_1),A' \cup B \big) \\[2 mm]
				&\le (1+\sigma)\left(\cF(y_1,P_1,A) + \cF(y_2,P_2,B)\right) \\[1 mm]
				& \quad + M_\sigma \int_{(A\cap B)\setminus A'} \left(|y_1-y_2|^2+|P_1-P_2|^q\right)\dd x + \sigma.
			\end{aligned}
		\end{equation}
	\end{definition}

	The use of the path $\gamma$ in \eqref{FE} is motivated
	by the simple observation that
	the ``convex combination'' $\varphi P_1+(1-\varphi)P_2$ does not belong
	to $\SL(3)$ in general.
	For example, if we let
	$$P_1 = \left(
	\begin{array}{ccc}
	1 & 0 & 0 \\
	0 & 1 & 0 \\
	0 & 0 & 1
	\end{array}
	\right), \quad P_2 = \left(
	\begin{array}{ccc}
	-1 & 0 & 0 \\
	0 & -1 & 0 \\
	0 & 0 & 1
	\end{array}
	\right),$$
	we see that
	$\det P_1 = \det P_2 = 1$, but $\det(P_1/2 + P_2/2) = 0$.
	However, not any path $\gamma$ is appropriate for our purposes,
	because in general it may not be completely contained in $K$, that is, in the domain of $H$. Moreover, the regularization term $\nabla P$ calls for some regularity of $\gamma$ with respect to its endpoints (see \eqref{eq:nablagamma}) and for a specific bound on the velocity $\dot\gamma$ (see \eqref{eq:gammadot}). To tackle such issues, in Proposition~\ref{stm:fundest} we select $\gamma$ suitably.
	
	The next statement grants that
	the class $\cC=\{\cF_k\}$ meets the definition above.
		
	\begin{proposition}\label{stm:fundest}
		The sequence of functionals
		$\{\cF_k\}_{k\in\nat}$ defined in \eqref{FepsA}
		satisfies uniformly the fundamental estimate \eqref{FE},
		upon choosing $\gamma$ as the map
		that associates	to $(t,F,G) \in [0,1] \times K \times K$
		the image at $t$ of the unique shortest path connecting $F$ and $G$.
	\end{proposition}
	
	\begin{proof}
		Fix $A,A',B \in \cA(\Omega)$ with $A' \Subset A$ and $\sigma > 0$. Fix also $(y_1,P_1),(y_2,P_2) \in W^{1,2}(\Omega;\real^3) \times W^{1,q}(\Omega;K)$.
		We need to choose suitably the cut-off function $\varphi$.
		
		Let $\delta \coloneqq {\rm dist}(A',\partial A)>0$.
		For a fixed $N \in \nat$, $N > 1$, we define
		\begin{align*}
			C_0 &\coloneqq A', \\
			C_j &\coloneqq \left\lbrace x \in A : {\rm dist}(x,A') < \frac{j}{N}\delta \right\rbrace, \quad j = 1,\dots,N, \\
			D_j &\coloneqq C_j \setminus C_{j-1}.
		\end{align*}
		Observe that for all $j$
		we can construct a cut-off function $\varphi_j$
		between $C_{j-1}$ and $C_j$ such that $|\nabla\varphi_j| \le 2N/\delta$.
		Now, setting \EEE $\gamma^{-1}(t,F,G)\coloneqq (\gamma(t,F,G))^{-1}$ for all $t\in[0,1]$,
		\begin{equation}\label{p4.1}
		\cF_k\left(\varphi_j y_1+(1-\varphi_j)y_2,\gamma\circ(\varphi_j,P_2,P_1),A' \cup B\right)
		= I_{\rm{el}}+I_{\hard}+I_{\rm{reg}},
		\end{equation}
		where
		\begin{gather*}
			I_{\rm el} \coloneqq \int_{A' \cup B}  
					W\left(\frac{x}{\eps_k},
						\big[\varphi_j \ny_1+(1-\varphi_j)\ny_2+\nabla\varphi_j\otimes(y_1-y_2)\big]
						\big(\gamma^{-1}\circ(\varphi_j,P_2,P_1)\big)
						\right) \! \dd x, \\
			I_\hard \coloneqq
				\int_{A' \cup B} 
					H\left(\frac{x}{\eps_k},\gamma\circ(\varphi_j,P_2,P_1)\right) \! \dd x, \\
			I_{\rm reg} \coloneqq \int_{A' \cup B}
				\left|\nabla \big(\gamma\circ(\varphi_j,P_2,P_1)\big)\right|^q \! \dd x.
		\end{gather*}
		Since $P_1,P_2 \in K$, \eqref{Pbound} holds.
		We now choose $\gamma$  as the map such that $\gamma(t,F,G)$ is the evaluation at $t$ of the unique shortest path connecting $F$ and $G$ for all $(t,F,G) \in [0,1] \times K \times K$.
		By Proposition~\ref{stm:Bao} we know that such $\gamma$ exists,
		that it lies completely in $K$ if $F,G\in K$
		(see \ref{H1}),
		and that it depends smoothly on its arguments.
		
		Let us estimate the three summands above separately.
		For the elastic contribution we have
		\begin{align*}
			I_\el
			&= \int_{(A' \cup B) \cap C_{j-1}} W\left(\frac{x}{\eps_k},\ny_1 P_1^{-1}\right) \dd x + \int_{(A' \cup B) \setminus C_j}  W\left(\frac{x}{\eps_k},\ny_2 P_2^{-1}\right) \dd x \nonumber\\
			& \quad + \int_{(A' \cup B) \cap D_j}\!\!  W\left(\frac{x}{\eps_k},\big[\varphi_j \ny_1+(1-\varphi_j)\ny_2+\nabla\varphi_j\otimes(y_1-y_2)\big]\big(\gamma^{-1}\circ(\varphi_j,P_2,P_1)\big)\right)\! \dd x \nonumber\\[2 mm]
			&\le \ia   W\left(\frac{x}{\eps_k},\ny_1 P_1^{-1}\right) \dd x + \int_B   W\left(\frac{x}{\eps_k},\ny_2 P_2^{-1}\right) \dd x \nonumber\\
			& \quad + c \int_{B \cap D_j}  \left(1 + |\ny_1|^2 + |\ny_2|^2 + \left(\frac{2N}{\delta}\right)^2|y_1-y_2|^2\right) \dd x,
		\end{align*}
		where we used the fact that $(A' \cup B) \cap C_{j-1} \subseteq C_{j-1} \subseteq A$, $(A' \cup B) \setminus C_j \subseteq B$, $(A' \cup B) \cap D_j = B \cap D_j$, the growth condition \ref{E-growth} and the uniform bound \eqref{Pbound}. 
		Analogously, for the second summand we find
		\begin{align*}
			I_\hard
			&= \int_{(A' \cup B) \cap C_{j-1}}   H\left(\frac{x}{\eps_k},P_1\right) \dd x + \int_{(A' \cup B) \setminus C_j}  H\left(\frac{x}{\eps_k},P_2\right) \dd x \nonumber\\
			& \quad + \int_{(A' \cup B) \cap D_j}   H\left(\frac{x}{\eps_k},\gamma\circ(\varphi_j,P_2,P_1)\right) \dd x \nonumber\\
			&\le \ia   H\left(\frac{x}{\eps_k},P_1\right) \dd x + \int_B   H\left(\frac{x}{\eps_k},P_2\right) \dd x + c\,\mathcal{L}^3(B \cap D_j),
		\end{align*}
		where we exploited again the geodesic convexity of $K$ and
		the fact that $H$ is bounded on $K$, since it is Lipschitz.
		
		We now estimate $I_{\rm reg}$.	
		The chain rule yields
		\begin{multline}\label{eq:nablagamma}
			\nabla \big(\gamma\circ(\varphi_j,P_2,P_1)\big)
			\\ = \big[ \dot\gamma \circ(\varphi_j,P_2,P_1) \big] \otimes \nabla \varphi_j +
			\big[ \partial_F \gamma \circ(\varphi_j,P_2,P_1) \big] \nabla P_2
			+  \big[ \partial_G \gamma \circ(\varphi_j,P_2,P_1) \big] \nabla P_1,
		\end{multline}
		whence, observing that by Proposition~\ref{stm:Bao} the two differentials \EEE
		$\partial_F \gamma$ and $\partial_G \gamma$ are continuous functions
		restricted to compact sets,
		\begin{align*}
			I_{\rm reg}
			&= \int_{(A' \cup B) \cap C_{j-1}} |\nP_1|^q \dd x + \int_{(A' \cup B) \setminus C_j} |\nP_2|^q \dd x
			+ \int_{(A' \cup B) \cap D_j} \left|\nabla \big(\gamma\circ(\varphi_j,P_2,P_1)\big) \right|^q \dd x \\
			&\le \ia |\nP_1|^q \dd x + \int_B |\nP_2|^q \dd x \\
			& \quad +
			\left(\frac{2N}{\delta}\right)^q \int_{B\cap D_j} | \dot\gamma \circ(\varphi_j,P_2,P_1) |^q \dd x + c \int_{B \cap D_j} \big(|\nP_1|^q + |\nP_2|^q\big) \dd x.
		\end{align*}
		We now resort to the explicit expression of $\gamma$ in terms of the exponential map, see \eqref{eq:shortest-exp}.
		From \eqref{Pbound} and recalling that by definition $\exp_{F}^{-1}(F) = 0$, it follows
		$$
		\begin{aligned}
		|\dot \gamma(t,F,G)|
		&= |\exp\big( F , t\exp_{F}^{-1}(G) \big) \exp_{F}^{-1}(G)| \\
		&\le c\, |\exp_{F}^{-1}(G)|
		= c\, |\exp_{F}^{-1}(G) - \exp_{F}^{-1}(F)|.
		\end{aligned}
		$$
		Since $\exp(F,\,\cdot\,)$ is a $C^1$-diffeomorphism around the origin (see p.\,\pageref{eq:shortest-exp}), $\exp_F^{-1}$ is Lipschitz in $K$, that is, for some $c>0$ it holds
		$$
		\Delta_I\Big(\exp_{F}^{-1}(G) - \exp_{F}^{-1}(F)\Big) \le c\,D(F,G),
		$$
		where $\Delta_I$ is the Minkowski norm on p.\,\pageref{D0} and $D$ is defined in \eqref{eq:D}.
		Therefore, recalling \ref{D2} and up to redefining the constant $c$, we obtain
		\begin{multline*}
		|\exp_{F}^{-1}(G) - \exp_{F}^{-1}(F)| \le c_4^{-1}\Delta_I\Big(\exp_{F}^{-1}(G) - \exp_{F}^{-1}(F)\Big) \\
		\le c\,D(F,G) \le c\,|F-G|, 
		\end{multline*}
		where the last inequality follows from Lemma~\ref{stm:DvsEucl}. We hence conclude
		\begin{equation}\label{eq:gammadot}
		|\dot \gamma(t,F,G)| \le c\,|F-G|,
		\end{equation}
		so that
		\begin{equation}\label{p4.4}
			\begin{aligned}
				I_{\rm reg} &\le \ia |\nP_1|^q \dd x + \int_B |\nP_2|^q \dd x \\
				& \quad + c \left[\left(\frac{2N}{\delta}\right)^q\int_{B \cap D_j} |P_1-P_2|^q \dd x + \int_{B \cap D_j} \big(|\nP_1|^q + |\nP_2|^q \big) \dd x \right].
			\end{aligned}
		\end{equation}
		By gathering \eqref{p4.1}--\eqref{p4.4} we obtain
		\begin{align*}
			&\cF_k\left(\varphi_j y_1+(1-\varphi_j)y_2,\gamma\circ(\varphi_j,P_2,P_1),A' \cup B\right) \\[2 mm]
			&\le \cF_k(y_1,P_1,A) + \cF_k(y_2,P_2,B) \\
			& \quad + c \int_{B \cap D_j}   \left(1 + |\ny_1|^2 + |\ny_2|^2 + \left(\frac{2N}{\delta}\right)^2|y_1-y_2|^2\right) \dd x \\
			& \quad + c\left[\mathcal{L}^3(B \cap D_j) + \left(\frac{2N}{\delta}\right)^q\int_{B \cap D_j} |P_1-P_2|^q \dd x + \int_{B \cap D_j} \big(|\nP_1|^q + |\nP_2|^q\big) \dd x\right].
		\end{align*}
		Now, note that
		\begin{align*}
			&\sum_{j=1}^{N} \int_{B \cap D_j}   \left(1 + |\ny_1|^2 + |\ny_2|^2 + |\nP_1|^q + |\nP_2|^q\right) \dd x \\
			&\le \int_{(A \setminus A') \cap B}   \left(1 + |\ny_1|^2 + |\ny_2|^2 + |\nP_1|^q + |\nP_2|^q\right) \dd x,
		\end{align*}
		thus there certainly exists $\ell\in {1,\dots,N}$ such that
		\begin{align*}
			&\int_{B \cap D_{\ell}}  \left(1 + |\ny_1|^2 + |\ny_2|^2 + |\nP_1|^q + |\nP_2|^q\right) \dd x \\
			&\le \frac1N \int_{(A \setminus A') \cap B}   \left(1 + |\ny_1|^2 + |\ny_2|^2 + |\nP_1|^q + |\nP_2|^q\right) \dd x \\
			&\le \frac1N \mathcal{L}^3((A \setminus A') \cap B) + \frac{c}{N} \Big(\cF_k(y_1,P_1,A) + \cF_k(y_2,P_2,B)\Big)
		\end{align*}
		where in the last inequality we exploited the fact that the set $(A \setminus A') \cap B$ is contained in both $A$ and $B$, together with the growth condition from below in \ref{E-growth}. Therefore, we obtain
		\begin{align*}
			&\cF_k\left(\varphi_{\ell} y_1+(1-\varphi_{\ell})y_2,\gamma\circ(\varphi_\ell,P_2,P_1),A' \cup B\right) \\[2 mm]
			&\quad\le \left(1 + \frac{c}{N}\right)\Big(\cF_k(y_1,P_1,A) + \cF_k(y_2,P_2,B)\Big) \\
			& \quad\quad + c \int_{(A\cap B)\setminus A'}\left[\left(\frac{2N}{\delta}\right)^2|y_1-y_2|^2+\left(\frac{2N}{\delta}\right)^q|P_1-P_2|^q\right]\! \dd x + \frac{2c}{N} \mathcal{L}^3((A \setminus A') \cap B).
		\end{align*}
		Finally, choosing $N$ such that
		$$
		\begin{cases}
			\displaystyle\frac{c}{N} < \sigma, \\[4 mm]
			\displaystyle \frac{2c\,\mathcal{L}^3((A \setminus A') \cap B)}{N} < \sigma,
		\end{cases}
		$$
		and letting
		$$M_\sigma \coloneqq c\left[\left(\frac{2N}{\delta}\right)^2 + \left(\frac{2N}{\delta}\right)^q\right],$$
		we see \eqref{FE} is satisfied, and the proof is complete.
	\end{proof}
	
	Still following the strategy outlined in Subsection~\ref{sec:Gammaconv},
	we analyze the properties of the $\Gamma(\tau)$-lower and $\Gamma(\tau)$-upper limit
	of our sequence $\cF_k$ when regarded as set functions.
	We recall that (see, e.g., \cite[Chapter~7]{BrDFr}),
	if $\{\eps_k\}_{k\in\nat}$ is such that $\eps_k \to 0^+$
	and if $(y,P,A) \in W^{1,2}(\Omega;\real^3)\times W^{1,q}(\Omega;\SL(3))\times \cA(\Omega)$,
	we have
	\begin{equation}\label{infsup}
		\begin{split}
			\cF'(y,P,A)
			& \coloneqq \Gamma(\tau)\mbox{-}\liminf_{k\to+\infty} \cF_k(y,P,A) \\
			& \coloneqq \inf \left\{
				\liminf_{k\to+\infty} \cF_k(y_k,P_k,A)
				: (y_k,P_k) \stackrel{\tau}{\to} (y,P)
				\right\}, \\
			\cF''(y,P,A)
			& \coloneqq \Gamma(\tau)\mbox{-}\limsup_{k\to+\infty} \cF_k(y,P,A) \\
			& \coloneqq \inf \left\{
					\limsup_{k\to+\infty} \cF_k(y_k,P_k,A)
					: (y_k,P_k) \stackrel{\tau}{\to} (y,P)
					\right\},
		\end{split}
	\end{equation}
	and that $\{\cF_k\}$ $\Gamma(\tau)$-converges to $\cF$
	if and only if $\cF=\cF'=\cF''$.
	A key-step is the following ``almost subadditivity'' result, which is a modification of \cite[Proposition 11.5]{BrDFr} or \cite[Proposition 18.3]{DalM}.
	It is a consequence of the fundamental estimate, but it does not depend on the explicit form of the functionals at stake.
	
	\begin{proposition}\label{almsub}
		Let $\{\cF_k\}_{k\in\nat}$, $\cF'$ and $\cF''$ be as above.
		Then for all $A,A',B \in \cA(\Omega)$ with $A' \Subset A$ and
		for all $(y,P) \in W^{1,2}(\Omega;\real^3)\times W^{1,q}(\Omega;K)$
		\begin{align*}
			\cF'(y,P,A' \cup B) &\le \cF'(y,P,A) + \cF''(y,P,B), \\
			\cF''(y,P,A' \cup B) &\le \cF''(y,P,A) + \cF''(y,P,B).
		\end{align*}
	\end{proposition}
	
	\begin{proof}
		We need to prove the inequalities only when
		$\cF'(y,P,A)$, $\cF''(y,P,A)$ and  $\cF''(y,P,B)$ are finite.
		By the definitions in \eqref{infsup}, there exist sequences $\{y'_k\}_k$, $\{P'_k\}_k$, $\{y''_k\}_k$, $\{P''_k\}_k$ such that
		\begin{gather*}
			(y'_k,P'_k) \stackrel{\tau}{\to} (y,P), \quad
			(y''_k,P''_k) \stackrel{\tau}{\to} (y,P),\\[3pt]
			\cF'(y,P,A) = \liminf_{k\to+\infty} \cF_k(y'_k,P'_k,A), \quad
			\cF''(y,P,B) = \limsup_{k\to+\infty} \cF_k(y''_k,P''_k,B).
		\end{gather*}
		Let us fix $\sigma > 0$. The fundamental estimate \eqref{FE} gives a constant $M_\sigma$,
		a sequence $\{\varphi_k\}_k$ of cut-off functions between $A'$ and $A$,
		and a sequence $\{\gamma_k\}_k$ of shortest paths from $P''_k$ to $P'_k$ such that
		\begin{equation}\label{p4.5}
			\begin{aligned}
				&\cF_k\left(\varphi_k y'_k+(1-\varphi_k)y''_k,\gamma_k\circ(\varphi_k,P_k'', P_k'),A' \cup B\right) \\[2 mm]
				&\le (1+\sigma)\big(\cF_k(y'_k,P'_k,A) + \cF_k(y''_k,P''_k,B)\big) \\[3pt]
				& \quad + M_\sigma \int_{(A\cap B)\setminus A'} \left(|y'_k-y''_k|^2+|P'_k-P''_k|^q\right)\dd x + \sigma.
			\end{aligned}
		\end{equation}
		Recalling \eqref{eq:shortest-exp}, we find
		$$
		\begin{array}{rcll}
			\varphi_k y'_k+(1-\varphi_k)y''_k &\to& y &\text{strongly in }L^2(A' \cup B;\real^3), \\[2 mm]
			\gamma_k\circ(\varphi_k,P_k'', P_k') &\to& P &\text{uniformly.}
		\end{array}
		$$
		Taking the lower limit in \eqref{p4.5} we obtain
		\begin{align*}
			\cF'(y,P,A' \cup B) &\le \liminf_{k\to+\infty} \cF_k\left(\varphi_k y'_k+(1-\varphi_k)y''_k,\gamma_k\circ(\varphi_k,P_k'',P_k'),A' \cup B\right) \\
			&\le (1+\sigma)\left(\liminf_{k\to+\infty} \cF_k(y'_k,P'_k,A) + \limsup_{k\to+\infty} \cF_k(y''_k,P''_k,B)\right) + \sigma \\
			&= (1+\sigma)\left(\cF'(y,P,A) + \cF''(y,P,B)\right) + \sigma,
		\end{align*}
		and the first inequality in the statement follows by letting $\sigma \to 0$. The second one is obtained by taking the upper limit in \eqref{p4.5} and arguing in a similar way.
	\end{proof}
	
	Building on Proposition~\ref{almsub},
	we next establish subadditivity and inner regularity.
	We exploit the fact that
	$\cF'(y,P,\,\cdot\,)$ and $\cF''(y,P,\,\cdot\,)$ are increasing set functions,
	as well as the growth condition
	\begin{equation}\label{eq:limsup-est}
		\cF''(y,P,A) \le c \ia \left(1 + |\ny P^{-1}|^2 + |\nP|^q\right) \dd x.
	\end{equation}
	
	\begin{proposition}[Subadditivity and inner regularity]\label{sub}
		Let $\{\cF_k\}_{k\in\nat}$, $\cF'$ and $\cF''$ be as before.
		For all $(y,P) \in W^{1,2}(\Omega;\real^3)\times W^{1,q}(\Omega;K)$
		\begin{enumerate}
			\item$\cF'(y,P,\,\cdot\,)$ and $\cF''(y,P,\,\cdot\,)$ are inner regular,
			\item $\cF''(y,P,\,\cdot\,)$ is subadditive.
		\end{enumerate}
	\end{proposition}
	\begin{proof}
		We prove the two statements separately.
		
		\noindent\textit{(1)}
		Let us focus on $\cF'$.
		According to Definition~\ref{setfunct} we need to prove that
		$$\cF'(y,P,A) = \sup\left\lbrace \cF'(y,P,B) : B \in \cA(\Omega), \, B \Subset A \right\rbrace.$$
		Since $\cF'(y,P,\,\cdot\,)$ is an increasing set function,
		it is trivial that 
		$\cF'(y,P,A)$ is larger than the supremum in the formula above.
		As for the reverse inequality, let us fix $A \in \cA(\Omega)$ and let $C \subset A$ be compact. There exist $B,B' \in \cA(\Omega)$ such that $C \subset B' \Subset B \Subset A$.
		Then $B' \cup (A\setminus C) = A$.
		We now apply Proposition~\ref{almsub}:
		\begin{align*}
			\cF'(y,P,A) & = \cF'(y,P,B' \cup (A\setminus C)) \\
			& \le \cF'(y,P,B) + \cF''(y,P,A\setminus C) \\
			&\le \sup\left\lbrace \cF'(y,P,B) : B \in \cA(\Omega), \, B \Subset A \right\rbrace + c \int_{A \setminus C} \left(1 + |\ny P^{-1}|^2 +  |\nP|^q\right) \dd x.
		\end{align*}
		By the arbitrariness of $C$, we can let $C$ invade $A$ and we obtain
		$$\cF'(y,P,A) \le \sup\left\lbrace \cF'(y,P,B) : B \in \cA(\Omega), \, B \Subset A \right\rbrace.$$
		
		The same argument applies to $\cF''$.
		
		\noindent\textit{(2)} Fix $A,B \in \cA(\Omega)$ and let $C \in \cA(\Omega)$
		satisfy $C \Subset A \cup B$. We find $A' \in \cA(\Omega)$ such that $A' \Subset A$ and that $C \subset A' \cup B$. Note that $\overline{C} \setminus B$ is compact and $\overline{C} \setminus B \subset A$. Now,
		by the monotonicity of $\cF''(y,P,\,\cdot\,)$ and Proposition~\ref{almsub},
		$$\cF''(y,P,C) \le \cF''(y,P,A' \cup B) \le \cF''(y,P,A) + \cF''(y,P,B).$$
		By taking the supremum over $C$, inner regularity yields
		$$\cF''(y,P,A \cup B)  \le \cF''(y,P,A) + \cF''(y,P,B).$$
	\end{proof}

By gathering the previous results,
we infer the existence of a $\Gamma$-limit admitting an integral representation.

\begin{proof}[Proof of Theorem~\ref{stm:homo-fin-plast-bis}]
Let $\{\eps_k\}_{k\in\nat}$ be such that $\eps_k \to 0^+$, and let $\cF'$ and $\cF''$ be as in \eqref{infsup}.
It follows from the definitions that they are increasing set functions,
and Proposition~\ref{sub} yields inner regularity for both and subadditivity for $\cF''$.
By standard arguments (see, e.g., \cite[Theorem~10.3]{BrDFr}),
we deduce that, upon extraction of a subsequences, 
$$\cF(y,P,A) \coloneqq \Gamma(\tau)\mbox{-}\lim_{k\to+\infty} \cF_k(y,P,A)$$
exists for all triple $(y,P,A)$.
From the definition of $\Gamma$-limit and from the fact that $\cF_k(y,P,\,\cdot\,)$ is a measure for each $k$,
it follows
that $\cF(y,P,\,\cdot\,)$ is superadditive.
We are then in a position to apply the De Giorgi-Letta criterion (see Subsection~\ref{sec:Gammaconv}),
which ensures that $\cF(y,P,\,\cdot\,)$ is the restriction of a Borel measure to $\cA(\Omega)$
for $(y,P)\in W^{1,2}(A;\real^3) \times W^{1,q}(A;K)$.
Thanks to \eqref{eq:limsup-est} we infer also that it is
absolutely continuous with respect to the Lebesgue measure.

Thus, by Radon-Nikodym theorem there exists $f \in L^1_{\rm loc}(\real^3)$ such that
$$\cF(y,P,A) = \ia f(x) \dd x.$$
\end{proof}

\subsection{Characterization of the limiting energy density}
So far, we have proved a $\Gamma(\tau)$-compactness result
for the functionals \eqref{FepsA}, and
we have shown that the limit is actually an integral functional.
We now provide an explicit formula for the limiting energy density.

\begin{proof}[Proof of Theorem~\ref{stm:homo-fin-plast}]
Let $y \in W^{1,2}(\Omega;\real^3)$, $P \in W^{1,q}(\Omega;K)$.
Our aim is to identify $f$ in \eqref{Fint} as
\begin{equation}\label{g}
	f(x) = W_\homo(\ny(x),P(x)) + H_\homo(P(x)) + |\nP(x)|^q \quad \text{ for a.\,e.\ } x \in \Omega.
\end{equation}
This also entails (see \cite[Proposition 7.11]{BrDFr}) that
the whole family $\{\cF_\eps\}$ $\Gamma(\tau)$-converges to $\cF$
as in the statement of the theorem

For all $x \in \real^3$, $F \in \matr$ and $G \in K$ let us define
$$\widetilde W(x,F,G) \coloneqq W(x,F G^{-1}).$$
From \ref{E1} and \ref{E-growth} we deduce that
\begin{enumerate}
	\item $\widetilde W$ is a Borel function and $\widetilde W(\,\cdot\,,F,G)$ is $Q$-periodic;
	\item $\widetilde W(x,\,\cdot\,,G)$ is $2$-coercive and has at most quadratic growth,
	uniformly in $x$ and $G$:
	there exist $0 < \tilde c_1 \le \tilde c_2$ such that
	for a.\,e.\ $x \in \real^3$, for all $F \in \matr$ and for all $G \in K$
	$$ \tilde c_1 |F|^2 \le \widetilde W(x,F,G) \le \tilde c_2\left(|F|^2+1\right).$$
\end{enumerate}
Note that the independence of $\tilde c_1$ and $\tilde c_2$ from $G$ is a consequence
of \eqref{stm:Pbound} and \eqref{Pbound}, respectively.

For $A\in\mathcal A(\Omega)$, we now introduce the functional
\begin{equation*}
	\cG_k(y,G,A) \coloneqq \ia \left[
		\widetilde W\left(\frac{x}{\eps_k},\ny(x),G\right)
		+
		H\left(\frac{x}{\eps_k},G\right)
	\right]\! \dd x
\end{equation*}
where $G\in K$ is fixed, and we observe that
$$
\cF_k(y,P,A) =
\ia \left[ \widetilde W\left(\frac{x}{\eps_k},\ny(x),P(x)\right) + H\left(\frac{x}{\eps_k},P(x)\right) + |\nP(x)|^q \right] \!\!\dd x.
$$
From Theorem~\ref{stm:homo-fin-plast-bis} we know that
$$
\cF(y,P,A)
\coloneqq \Gamma(\tau)\mbox{-}\lim_{k\to+\infty} \cF_k(y,P,A)
= \ia f(x) \dd x,$$
while, if $x_0\in A$ and $P_0 \coloneqq P(x_0)$,
Theorem~\ref{stm:hom-classic} and Riemann-Lebesgue lemma on rapidly oscillating functions yield
\begin{equation}\label{convG}
	\begin{aligned}
		\cG(y,P_0,A)
		&\coloneqq \Gamma(L^2)\mbox{-}\lim_{k\to+\infty} \cG_k(y,P_0,A) \\
		&= \ia \Big(W_\homo \big( \ny(x),P(x_0) \big) + H_\homo(P(x_0))\Big) \dd x.
	\end{aligned}
\end{equation}
We next show that \eqref{g} holds by exploiting these two convergence results.

From now on, let $x_0\in\Omega$ be a Lebesgue point of the functions $f$, $\ny$ and $\nP$.
Notice that almost every $x_0\in\Omega$ has such property.
Let $A\in \mathcal{A}(\Omega)$ contain $x_0$,
and let $\{(y_k,P_k)\}\subset W^{1,2}(A;\real^3) \times W^{1,q}(A;K)$ be a generic sequence.
We use \ref{E-lip} and \eqref{Pbound} to estimate the difference
between the elastic contribution in $\cF_k(y_k,P_k,A)$ and $\cG_k(y_k,P_0,A)$:
\begin{align*}
	&\ia \left| 
	\widetilde W \left(\frac{x}{\eps_k},\ny_k,P_k\right)
	- \widetilde W \left(\frac{x}{\eps_k},\ny_k,P_0\right)
	\right|\!\! \dd x \nonumber\\
	&\qquad \le c_3 \ia \Big(
	1 + |\ny_k P_k^{-1}| + |\ny_k P^{-1}_0| \Big)
	|\ny_k| \left|P_k^{-1} - P^{-1}_0\right|
	\!\!\dd x \nonumber\\
	&\qquad \le c_3  \ia
	\big(1 + 2c_K |\ny_k| \big) |\ny_k|
	\Big(\left|P_k^{-1} - P^{-1}\right| + \left|P^{-1} - P^{-1}_0)\right|\Big)
	\! \dd x \nonumber\\
	&\qquad \le c\,
	\Big( \mathcal{L}^3(A)+ \| \nabla y_k\|^2_{L^2(A)}\Big) 
	\Big( \| P_k^{-1} - P^{-1} \|_{L^\infty(A)} +  \| P^{-1} - P^{-1}_0 \|_{L^\infty(A)}\Big)	,
\end{align*}
where $c$ is a constant independent of $k$ and $A$.
We now use \ref{H2} to estimate the difference
between the hardenings in $\cF_k(y_k,P_k,A)$ and $\cG_k(y_k,P_0,A)$:
\begin{align*}
	\ia \left|H\left(\frac{x}{\eps_k},P_k\right) - H\left(\frac{x}{\eps_k},P_0\right) \right| \dd x \le c \, \mathcal{L}^3(A) \Big( \| P_k - P \|_{L^\infty(A)} +  \| P - P_0 \|_{L^\infty(A)}\Big).
\end{align*}
By the definitions of $\cF_k$ and $\cG_k$
we obtain
\begin{equation}\label{distG}
	\begin{split}
		&\left| \cF_k(y_k,P_k,A) - \ia |\nP_k|^q \dd x-\cG_k(y_k,P_0,A)\right|
		\\ &\qquad\leq c\,
		\Big( \mathcal{L}^3(A)+  \| \nabla y_k\|^2_{L^2(A)} \Big)
		\Big( \| P_k^{-1} - P^{-1} \|_{L^\infty(A)} + \| P^{-1} - P^{-1}_0 \|_{L^\infty(A)}\Big) \\
		&\qquad\quad + c \, \mathcal{L}^3(A) \Big( \| P_k - P \|_{L^\infty(A)} +  \| P - P_0 \|_{L^\infty(A)}\Big).
	\end{split}
\end{equation}

We now prove that
\begin{equation}\label{eq:f>Whom}
	f(x) \ge W_\homo\big( \ny(x),P(x) \big) + H_\homo(P(x)) + |\nP(x)|^q \quad \text{ for a.\,e.\ } x \in \Omega.
\end{equation}
To this aim, we select a recovery sequence
$\{(y_k,P_k)\}\subset W^{1,2}(A;\real^3) \times W^{1,q}(A;K)$
for $\cF(y,P,A)$, namely
\begin{gather}
	(y_k,P_k) \stackrel{\tau}{\to} (y,P), \quad
	\lim_{k\to+\infty} \cF_k(y_k,P_k,A) = \cF(y,P,A). \label{eq:rec-tildeF}
\end{gather}
Owing to the growth assumptions on the energy densities,
we can without loss of generality suppose that
\begin{gather*}
	y_k \rightharpoonup y \quad\text{weakly in } W^{1,2}(A;\real^3),
	\qquad
	P_k \rightharpoonup P \quad\text{weakly in } W^{1,q}(A;K).
\end{gather*}
From \eqref{distG} coupled with the $2$-coercivity of $\widetilde W$,
recalling Remark~\ref{stm:conv-inv} about the convergence of inverses,
\eqref{convG} and \eqref{eq:rec-tildeF}, we infer
\begin{align*}
	\cF(y,P,A) &= 
	\lim_{k\to+\infty} \cF_k (y_k,P_k,A) \\
	& \geq \liminf_{k\to+\infty} \cG_k(y_k,P_0,A) +\liminf_{k\to+\infty} \ia |\nP_k|^q \dd x\\
	&\quad - c\, \| P^{-1} - P^{-1}_0 \|_{L^\infty(A)}
	\Big( \mathcal{L}^3(A)
	+ \lim_{k\to+\infty} \cF_k(y_k,P_k,A)
	\Big) - c \, \mathcal{L}^3(A) \,  \| P - P_0 \|_{L^\infty(A)} \\
	&\geq \cG(y,P_0,A) +\ia |\nP|^q \dd x \\
	& \quad	-c\,  \| P^{-1} - P^{-1}_0 \|_{L^\infty(A)}
	\Big( \mathcal{L}^3(A) + \cF(y,P,A) \Big) - c \, \mathcal{L}^3(A) \,  \| P - P_0 \|_{L^\infty(A)}.
\end{align*}
Observe that
the lower bound on $\cF(y,P,A)$ 
that we have just proved
holds for any set $A$ containing $x_0$.
In particular,
we select $A= B_r(x_0)$ and 
divide both sides of the corresponding estimates by $\mathcal{L}^3(B_r(x_0))$.
By letting $r \to 0$,
we apply Lebesgue differentiation theorem
to deduce \eqref{eq:f>Whom} for $x=x_0$.
Indeed,
\begin{equation*}
	\lim_{r \to 0} \| P^{-1} - P^{-1}_0 \|_{L^\infty(B_r(x_0))}
	\left( 1 + \frac{1}{\mathcal{L}^3(B_r(x_0))}\cF(y,P,B_r(x_0))\right) =0,
\end{equation*}
because \eqref{eq:limsup-est} grants
\[
\cF(y,P,A) \le c \ia \left(1 + |\ny P^{-1}|^2 + |\nP|^q\right) \dd x.
\]

Eventually, we are only left to prove that
\begin{equation}\label{eq:f<Whom}
	f(x) \le W_\homo\big(\ny(x),P(x)\big) + H_\homo(P(x)) + |\nP(x)|^q \quad \text{ for a.\,e.\ } x \in \Omega.
\end{equation}
To establish the estimate, we let $\{y_k\}\subset W^{1,2}(A;\real^3)$
be a recovery sequence for $\cG(y,P_0,A)$, that is,
\[
\lim_{k\to+\infty} \cG_k(y_k,P_0,A) = \cG(y,P_0,A),
\]
with $ y_k \rightharpoonup y$ weakly in $W^{1,2}(A;\real^3)$.
We consider estimate \eqref{distG}
for this particular $\{y_k\}$ and the constant sequence given by $P_k = P$ for all $k$.
Exploiting this time the $2$-coercivity of $\cG_k$ we obtain
\begin{align*}
	\cF(y,P,A) &\le
	\limsup_{k\to+\infty} \cF_k\left(y_k,P,A\right) \\
	&\le \lim_{k\to+\infty} \cG_k\big(y_k,P_0,A\big) + \ia |\nP|^q \dd x \\
	& \quad + c\, \| P^{-1} - P^{-1}_0 \|_{L^\infty(A)}
	\Big( \mathcal{L}^3(A)
	+ \lim_{k\to+\infty} \cG_k(y_k,P_0,A)
	\Big)  + c \, \mathcal{L}^3(A) \,  \| P - P_0 \|_{L^\infty(A)} \\
	&\le \cG(y,P_0,A) + \ia |\nP|^q \dd x \\
	& \quad + c \, \| P^{-1} - P^{-1}_0 \|_{L^\infty(A)} \Big(\mathcal{L}^3(A) + \cG\big(y,P_0,A\big)\Big) + c \, \mathcal{L}^3(A) \,  \| P - P_0 \|_{L^\infty(A)}.
\end{align*}

Arguing as above with $A = B_r(x_0)$ and exploiting the 2-growth conditions of $W_\homo$, we conclude that \eqref{eq:f<Whom} holds for $x = x_0$.

Since we can select almost every point in $\Omega$ as $x_0$, the conclusion follows from \eqref{eq:f>Whom} and \eqref{eq:f<Whom}.
\end{proof}

\subsection{Convergence of minimum problems}

We conclude this section with the proof of equicoercivity and convergence of (almost) minimizers, that is, Corollary~\ref{cor:comp}.
\begin{proof}[Proof of Corollary~\ref{cor:comp}]
	\textit{(i)} Let $\{(y_k,P_k)\}$ be a bounded energy sequence such that $\|y_k\|_{L^2} \le c$, uniformly in $k$. From the definition of $\cF_k$, for all $k \in \nat$
	\begin{equation}\label{eq:Lq_gradP}
	\|\nabla P_k\|_{L^q} \le c.
	\end{equation}
	Besides, for all $k$, assumption \ref{H1} and the uniform bound \eqref{Pbound} yield
	\begin{equation}\label{eq:Linfty_P}
	\|P_k\|_{L^\infty} + \|P_k^{-1}\|_{L^\infty} \le c.
	\end{equation}
	As for the elastic part, assumption \ref{E-growth} entails
	$$
	\|\nabla y_k P_k^{-1}\|_{L^2} \le c,
	$$
	which, together with \eqref{stm:Pbound} and \eqref{eq:Linfty_P}, gives
	\begin{equation}\label{eq:L2_y}
	\|\nabla y_k\|_{L^2} \le c.
	\end{equation}
	This and the uniform bound on the $L^2$-norm of $y_k$ imply that there exists $y \in W^{1,2}(\Omega;\real^3)$ such that
	$$
	y_k \to y \quad\mbox{strongly in } L^2(\Omega;\real^3)
	$$
	up to subsequences. From bounds \eqref{eq:Lq_gradP}--\eqref{eq:Linfty_P} and Morrey's embedding, we deduce the existence of $P \in W^{1,q}(\Omega;K)$ such that
	$$
	P_k \to P \quad\mbox{uniformly},
	$$
	again up to subsequences. Note that the uniform convergence of $\{P_k\} \subset W^{1,q}(\Omega;K)$ yields that $P$ attains values in $K$ as well. Recalling the definition of the topology $\tau$ in \eqref{eq:tau}, we have thus proved that $(y_k,P_k) \stackrel{\tau}{\to} (y,P)$.
	\\[10pt] \indent
	\textit{(ii)} By combining (i) with Theorem~\ref{stm:homo-fin-plast}, the result follows in a standard way, see, e.g., \cite[Theorem~7.2]{BrDFr}.
\end{proof}


\addtocontents{toc}{\protect\setcounter{tocdepth}{1}}

\section*{Acknowledgements}
We acknowledge support
from the Austrian Science Fund (FWF) projects \href{https://doi.org/10.55776/F65}{10.55776/F65}, \href{https://doi.org/10.55776/V662}{10.55776/V662}, \href{https://doi.org/10.55776/Y1292}{10.55776/Y1292},
from the FWF-GA\v{C}R project \href{https://doi.org/10.55776/I4052}{10.55776/I4052} (19-29646L), and  
from the OeAD-WTZ project CZ04/2019 (M\v{S}MT\v{C}R 8J19AT013).

\medskip

This version of the article has been accepted for publication, after peer review (when applicable) but is not the Version of Record and does not reflect post-acceptance improvements, or any corrections. The Version of Record is available online at: \url{http://dx.doi.org/10.1007/s00526-024-02673-0}.

%
%

\end{document}